\documentclass[11pt,oneside]{amsart}

\usepackage{amsmath,ifthen, amsfonts, amssymb, amsopn, color, enumerate, srcltx}

\newcounter{intronum}

\newtheorem{thm}{Theorem}[section]
\newtheorem{lem}[thm]{Lemma}

\newtheorem{prop}[thm]{Proposition}
\newtheorem{thmi}[intronum]{Theorem}
\newtheorem*{claim}{Claim}

\theoremstyle{definition}
\newtheorem{defn}[thm]{Definition}
\newtheorem{rem}[thm]{Remark}
\newtheorem{constr}[thm]{Construction}

\DeclareMathOperator{\dimension}{dim}
\DeclareMathOperator{\Aut}{Aut}
\DeclareMathOperator{\stabilizer}{Stab}
\newcommand{\neb}{\mathcal N}
\newcommand{\homology}{\ensuremath{{\sf{H}}}}
\newcommand{\field}[1]{\mathbb{#1}}
\newcommand{\integers}{\ensuremath{\field{Z}}}
\newcommand{\reals}{\ensuremath{\field{R}}}
\newcommand{\Euclidean}{\ensuremath{\field{E}}}

\let\oldmarginpar\marginpar
\renewcommand\marginpar[1]{\-\oldmarginpar[\raggedleft\footnotesize #1]%
{\raggedright\footnotesize #1}}

\newcommand{\cuco}{{\mathcal{X}}}

\begin{document}
\title[Cocompactly cubulated graph manifolds]{Cocompactly cubulated graph manifolds}

\author[M.F.~Hagen]{Mark F. Hagen$^\ast$}
\address{Dept. of Math., University of Michigan, Ann Arbor, MI, USA }
\email{markfhagen@gmail.com}
\thanks{$\ast$ This material is based upon work partially supported by the National Science Foundation
under Grant Number NSF 1045119.}

\author[P. Przytycki]{Piotr Przytycki$^\dag$}
\address{Institute of Mathematics, Polish Academy of Sciences, Warsaw, Poland}
\email{pprzytyc@mimuw.edu.pl}
\thanks{$\dag$ Partially supported by MNiSW grant N201 012 32/0718, the Foundation for Polish Science, and National Science Centre DEC-2012/06/A/ST1/00259.}

%\date{\today}
%\subjclass[2010]{Primary: 20F65; Secondary: 57M99}
%\keywords{graph manifold, CAT(0) cube complex, flip manifold, charge}
\maketitle

\begin{abstract}
Let $M$ be a graph manifold. We show that $\pi_1M$ is the fundamental group of a compact nonpositively curved cube complex if and only if $M$ is chargeless. We also prove that in that case $\pi_1M$ is virtually compact special.
\end{abstract}

\section{Introduction}\label{sec:introduction}
A \emph{graph manifold} is a compact oriented aspherical $3$--manifold $M$ that has only Seifert-fibred blocks in its JSJ decomposition. We say that a torsion-free group is (\emph{cocompactly}) \emph{cubulated}  if it is the fundamental group of a (compact) nonpositively curved cube complex. A (cocompactly) cubulated group is (\emph{compact}) \emph{special} if the complex is (compact) \emph{special}, i.e.\ admits a local isometry into the Salvetti complex of a right-angled Artin group.

Liu proved in~\cite{Liu:cube_GraphManifold} that if a graph manifold $M$ admits a nonpositively curved Riemannian metric, then $\pi_1M$ is virtually cubulated (and in fact special). Under the stronger hypothesis that $M$ has nonempty boundary, the same conclusion was obtained in~\cite{PrzytyckiWise:cube_GraphManifold}. However, the resulting cube complex was in general not compact.

The main goal of this paper is to answer Question~9.4 of Aschenbrenner, Friedl, and Wilton \cite{AFW} by characterizing graph manifolds $M$ with $\pi_1M$ virtually cocompactly cubulated, i.e.\ having a finite-index subgroup that acts freely and cocompactly on a CAT(0) cube complex. We show, moreover, that whenever this is the case, $\pi_1M$ is virtually compact special.

We note that if $M$ has no JSJ tori, i.e.\ $M$ is Seifert-fibred, then by \cite[Thm~6.12]{BridsonHaefliger:book} the group $\pi_1M$ is cocompactly cubulated if and only if the Euler number of the Seifert fibration vanishes. In this situation, the cube complex can easily be seen to be virtually special using \cite{Sco}.
If $M$ is a Sol manifold, then $\pi_1M$ is not cocompactly cubulated

Throughout the article, we therefore assume that $M$ is not a Sol manifold and has at least one JSJ torus, so that its underlying graph $\overline\Gamma=(\overline V,\overline E)$ has at least one edge. We also assume that $M$ does not contain $\pi_1$--injective Klein bottles, so that the base orbifolds of all Seifert-fibred blocks are oriented and hyperbolic. For each $\bar v\in \overline V$, we denote by $B_{\bar v}\subset M$ the corresponding Seifert-fibred block, and for each edge $\bar e\in \overline E$, we denote by $T_{\bar e}$ the corresponding JSJ torus.
For an edge $\bar e$ incident to $\bar v$, let $Z^{\bar e}_{\bar v}\subset T_{\bar e}$ be an embedded circle that is a fiber in $B_{\bar v}$.

\begin{defn}
\label{def:chargeless}
A graph manifold $M$ is \emph{chargeless} if for every block $B_{\bar v}$ we can assign integers $n_{\bar e}$ to all edges $\bar e=(\bar v, \bar v')$,
so that in integral homology $\homology_1(B_{\bar v})$ we have $\sum_{\bar e} n_{\bar e}[Z^{\bar e}_{\bar v'}]=0$.
\end{defn}

In other words, a graph manifold is chargeless if in each block there is a horizontal surface whose boundary circles are vertical in adjacent blocks.
Note that if $M'$ is a finite cover of a graph manifold $M$, then $M'$ is chargeless if and only if $M$ is chargeless. Our first result is the following.

\begin{thmi}\label{thmi:compact_special}
Let $M$ be a chargeless graph manifold. Then $\pi_1M$ is virtually compact special.
\end{thmi}

Our main theorem is the following converse.

\begin{thmi}\label{thmi:cocompact_cubulation}
Let $M$ be a graph manifold. If $\pi_1M$ is virtually cocompactly cubulated, then $M$ is chargeless.
\end{thmi}

Theorem~\ref{thmi:cocompact_cubulation} is one of few results giving an obstruction to being cocompactly cubulated for a specific class of groups. Another notable result of this type is Wise's characterization of tubular groups that are cocompactly cubulated \cite[Thm~5.8]{Wise:Tubular}.

\medskip
\noindent \textbf{Organization.} In Section~\ref{sec:flip_implies_cocompact} we introduce notation used to study graph manifolds. We construct an efficient family of surfaces in a graph manifold. The CAT(0) cube complex dual to the resulting system of walls in the universal cover is cocompact, as required in Theorem~\ref{thmi:compact_special}. In Section~\ref{sec:flat_construct} we study how images of isometric embeddings of $\Euclidean^2$ in a CAT(0) cube complex may intersect hyperplanes. We describe intersection patterns that allow one to recognize a combinatorial flat or a half-flat of dimension $>2$ in the complex. In Section~\ref{sec:cocompact_implies_flip} we show that such objects are not allowed in cocompact cubulations of graph manifolds. This enables us to find cocompact cubical convex cores for $\pi_1B_{\bar v}$. We then use Caprace--Sageev's rank-rigidity result to limit the possible cubulations of the blocks. The cubulations of adjacent blocks interact along what we call \emph{straight subgroups}, and analyzing them is the last step in the proof of Theorem~\ref{thmi:cocompact_cubulation}.

\medskip
\noindent \textbf{Prerequisities.} We assume basic knowledge of CAT(0) cube complexes and hyperplanes. We consider two metrics: the combinatorial metric on the $1$--skeleton with combinatorial geodesics and isometric embeddings on $1$--skeleta; and the CAT(0) metric on the entire complex, with CAT(0) geodesics and isometrically embedded copies $\mathcal E$ of $\Euclidean^2$. For an introduction to CAT(0) cube complexes, we recommend Sageev's survey article~\cite{SagPCMI}.

\medskip
\noindent \textbf{Acknowledgements.} We thank Henry Wilton and Daniel T. Wise for encouragement and Jordan Sahattchieve for discussions. M.F.H. thanks the Institute of Mathematics of the Polish Academy of Sciences for its hospitality during the period in which most of this work was completed.

\section{Chargeless graph manifolds}\label{sec:flip_implies_cocompact}
In this section we prove Theorem~\ref{thmi:compact_special}.

\subsection{Preliminaries}

Let $M$ be a compact oriented irreducible $3$--manifold that is not a Sol manifold and does not contain $\pi_1$--injective Klein bottles. The manifold $M$ contains a minimal collection of incompressible tori, called \emph{\textnormal{JSJ} tori}, all of whose complementary components, called \emph{blocks} are Seifert-fibred or atoroidal \cite[Thm~3.4]{Bonahon}. This decomposition is unique up to isotopy. We say that $M$ is a \emph{graph manifold} if all of the blocks are Seifert-fibred, and there is at least one JSJ torus.

The \emph{underlying graph} $\overline\Gamma=(\overline V,\overline E)$ of $M$ is the graph dual to the JSJ decomposition.
For each $\bar v\in V$, we denote by $B_{\bar v}$ the corresponding block, and for each edge $\bar e\in \overline E$, we denote by $T_{\bar e}$ the corresponding JSJ torus.
Let $F_{\bar v}$ be the base orbifold of $B_{\bar v}$, which is oriented and hyperbolic.

Let $G=\pi_1M$, which we identify with the group of covering transformations of the universal cover $\widetilde M$ of $M$. Let $\Gamma=(V,E)$ be the underlying Bass--Serre tree of $\widetilde M$ arising from the JSJ decomposition of $M$. The vertices $v\in V$ of $\Gamma$ correspond to the components $\widetilde B_v\subset \widetilde M$ of the preimages of blocks in $M$, which we also call \emph{blocks}. The edges $e\in E$ of $\Gamma$ correspond to the components $\widetilde T_e\subset \widetilde M$ of the preimages of JSJ tori in $M$, which we call \emph{\textnormal{JSJ} planes}. We denote by $G_v< G$ the stabilizer of the block $\widetilde B_v\subset\widetilde M$ and by $G_e< G$ the stabilizer of the JSJ plane $\widetilde T_e$.

An immersed surface in a block is \emph{horizontal} if it is transverse to the fibers and \emph{vertical} if it is a union of fibers.

\subsection{Efficient collection}
We will cubulate a chargeless graph manifold using a family of surfaces of the following type.

\begin{defn}[Turbine collection]
Let $M$ be a chargeless graph manifold. A \emph{turbine} collection $\mathcal{S}^\mathrm{tur}$ is a collection of surfaces that are immersed in $M$ and have the following form. For each block $B_{\bar v}$ consider an embedded surface $S_{\bar v}\subset B_{\bar v}$ such that $\partial S_{\bar v}\cap T_{\bar e}$ is the union of $n_{\bar e}$ parallel copies of $Z^{\bar e}_{\bar v'}$ from Definition~\ref{def:chargeless}. Let $A_{\bar e}$ be a non-$\partial$--parallel vertical annulus in $B_{\bar v'}$, both of whose boundary components are homotopic to $Z^{\bar e}_{\bar v'}$ in $T_{\bar e}$. Gluing $2n_{\bar e}$ parallel copies of the annuli $A_{\bar e}$ with $2$ parallel copies of $S_{\bar v}$ yields a surface $S^\mathrm{tur}_{\bar v}$ properly immersed in $M$. We let $\mathcal{S}^\mathrm{tur}$ denote the collection of $S^\mathrm{tur}_{\bar v}$ for all $\bar v$, together with some non-$\partial$--parallel vertical annuli $A$ with $\partial A\subset T$ for all boundary tori $T$ of $M$.
\end{defn}

\begin{rem}\label{rem:turbine_embed}
Let $\Gamma'\rightarrow\overline \Gamma$ be a cover of the underlying graph $\overline \Gamma$ of $M$. Let $M'\rightarrow M$ be the induced cover of $M$, i.e.\ the graph of spaces whose underlying graph is $\Gamma'$ and whose vertex spaces are isomorphic to the various blocks of $M$. If $\Gamma'$ is simple, i.e.\ has no loops or double edges, then every surface in $\mathcal S^\mathrm{tur}$ lifts to an embedding in $M'$.
\end{rem}

\begin{defn}[Torus collection, efficient collection]
Each block $B_{\bar v}$ of a graph manifold $M$ has a finite cover $B'_{\bar v}$ that is a product of a circle with a surface $F'_{\bar v}$. A \emph{torus} collection $\mathcal S_{\bar v}^\mathrm{tor}$ is a family of vertical tori in $B'_{\bar v}$ whose base curves \emph{fill} $F'_{\bar v}$. With respect to some hyperbolic metric on $F'_{\bar v}$ this means that the complementary components of the union of the geodesic representatives of the base curves are discs or annular neighborhoods of the boundary. Each torus $T\in\mathcal S^{\mathrm{tor}}_{\bar v}$ is equipped with a map to $M$ factoring through $B'_{\bar v}$. Let $\mathcal S^\mathrm{tor}$ be the union of the tori in all $\mathcal S_{\bar v}^\mathrm{tor}$ together with all the JSJ tori of $M$.

The union $\mathcal{S}^\mathrm{eff}=\mathcal{S}^\mathrm{tur}\cup \mathcal{S}^\mathrm{tor}$ of a turbine collection and a torus collection for a chargeless graph manifold is called an \emph{efficient} collection. We also assume that $\mathcal S^\mathrm{eff}$ is in general position.
Let $\widetilde{\mathcal{S}}^\mathrm{eff}$ be the collection of all connected surfaces in the universal cover $\widetilde{M}$ of $M$ covering the surfaces of $\mathcal{S}^\mathrm{eff}.$
\end{defn}

Each surface in $\widetilde{\mathcal{S}}^\mathrm{eff}$ cuts $\widetilde{M}$ into two halfspaces and the collection of such pairs endows $\widetilde{M}$ with a Haglund--Paulin wallspace structure (see e.g.\ \cite[Sec~2.1]{HruskaWise}).

\begin{constr}
\label{constr:Sageev}
Sageev's construction yields the \emph{\textnormal{CAT(0)} cube complex $\cuco$ dual to $\widetilde{\mathcal{S}}^\mathrm{eff}$}.
A vertex $x$ of $\cuco$ is a choice of a closed halfspace $x(S)$ for each $S\in\widetilde{\mathcal{S}}^\mathrm{eff}$ with $x(S)\cap x(S')\neq \emptyset$ for all $S,S'\in\widetilde{\mathcal{S}}^\mathrm{eff}$. Cubes of dimension $k$ are spanned by the sets of $2^k$ vertices differing on $k$ surfaces $S$. If that complex is disconnected, we restrict to the connected component containing the vertices $x$ for which there is a point $m\in \widetilde M$ with $m\in x(S)$ for all $S$.

The group $\pi_1M$ acts naturally on $\cuco$. Hyperplanes of $\cuco$ correspond to surfaces in $\widetilde{\mathcal{S}}^\mathrm{eff}$ and their stabilizers coincide. Maximal cubes of $\cuco$ correspond to maximal collections of pairwise intersecting surfaces in  $\widetilde{\mathcal{S}}^\mathrm{eff}$.
\end{constr}
Theorem~\ref{thmi:compact_special} is an immediate consequence of the following.

\begin{prop}
\label{prop:flip->cocompact}
Let $M$ be a chargeless graph manifold with an efficient collection $\mathcal{S}^\mathrm{eff}.$ Let $\cuco$ be the \textnormal{CAT(0)} cube complex dual to $\widetilde{\mathcal{S}}^\mathrm{eff}.$ Then the action of $\pi_1M$ on $\mathcal{X}$ is free and cocompact. Moreover, $\pi_1M\backslash \cuco$ is virtually special.
\end{prop}

\begin{lem}
\label{lem:flip->far}
Let $S,S'$ be distinct surfaces in $\widetilde{\mathcal{S}}^\mathrm{eff}$ intersecting a block $\widetilde{B}_v\subset\widetilde{M}$. If $S\cap S'\neq \emptyset$, then $S\cap S'\cap \widetilde{B}_v\neq\emptyset$. Moreover, at most one of $S\cap \widetilde{B}_v,S'\cap \widetilde{B}_v$ is horizontal.
\end{lem}

\begin{proof}
If one of $S\cap \widetilde{B}_v,S'\cap \widetilde{B}_v$ is horizontal and the other vertical, then they intersect.

If they are both horizontal, then they cover the same unique surface in $\mathcal{S}^\mathrm{eff}$ that has horizontal intersection with the block of $M$ covered by $\widetilde{B}_v$. By Remark~\ref{rem:turbine_embed} applied to the universal cover $\Gamma'\rightarrow\overline \Gamma$, through which the map $\Gamma\rightarrow \overline{\Gamma}$ factors, the surfaces $S,S'$ are disjoint.

If $S\cap \widetilde{B}_v,S'\cap \widetilde{B}_v$ are both vertical and disjoint, but $S$ and $S'$ intersect, then there is an adjacent block $\widetilde{B}_{v'}$ with both $S\cap \widetilde{B}_{v'},S'\cap \widetilde{B}_{v'}$ nonempty and horizontal. As before, we obtain that $S$ and $S'$ are disjoint, which is a contradiction.
\end{proof}

\begin{proof}[Proof of Proposition~\ref{prop:flip->cocompact}]
The action is free by \cite[Sec~3]{PrzytyckiWise:cube_GraphManifold} since $\mathcal{S}^\mathrm{eff}$ is \emph{sufficient} in the sense of \cite[Def~1.4]{PrzytyckiWise:cube_GraphManifold}.

For cocompactness it suffices to show that, for any collection of pairwise intersecting surfaces $\mathcal{S}\subset \widetilde{\mathcal{S}}^\mathrm{eff}$, there is a point of $\widetilde M$ uniformly close to each $S\in\mathcal{S}$. First note that by the Helly property for trees, there is a block $\widetilde{B}_v\subset \widetilde M$ intersecting all the surfaces in $\mathcal{S}$.
By Lemma~\ref{lem:flip->far} the surfaces $S\cap \widetilde{B}$ for $S\in\mathcal{S}$ pairwise intersect. Moreover, the collection $\mathcal{S}$ contains at most one surface $S^\mathrm{hor}$ with $S^\mathrm{hor}\cap \widetilde{B}_v$ horizontal and we can assume without loss of generality that there is such a surface in $\mathcal{S}$. Then the family of lines $S\cap S^\mathrm{hor}$ with $S\in \mathcal{S}-\{S^\mathrm{hor}\}$ is a family of uniform quasi-geodesics in the quasi-tree $S^\mathrm{hor}\cap \widetilde{B}_v$, pairwise at finite distance. Then, for example by \cite[Lem~3.4]{Sageev97}, there is a point in $S^\mathrm{hor}\cap \widetilde{B}_v$ uniformly close to all these lines.  Hence $G=\pi_1M$ acts cocompactly on $\cuco$.

We verify specialness using \cite[Thm~9.19]{HaglundWise:special}, saying that for a free and cocompact action of a group $G$ on a \textnormal{CAT(0)} cube complex $\cuco$ the following are equivalent:
\begin{enumerate}
 \item $G$ has a finite-index subgroup $G'$ such that $G'\setminus\cuco$ is special.
 \item For each hyperplane $\mathfrak h$ of $\cuco$, the subgroup $\stabilizer \mathfrak h\leq G$ is separable, and for intersecting hyperplanes $\mathfrak h,\mathfrak h'$, the double coset $\stabilizer \mathfrak h\, \stabilizer \mathfrak h'\subset G$ is separable.
\end{enumerate}

Each hyperplane in $\cuco$ has stabilizer $\stabilizer S$ for some $S\in\widetilde{\mathcal S}^\mathrm{eff}.$
Since $\mathcal S^\mathrm{eff}$ is in general position, each $\stabilizer S$ coincides with a conjugate of the fundamental group of a surface in $\mathcal{S}^\mathrm{eff}$. Since the surfaces in $\mathcal{S}^\mathrm{eff}$ are virtually embedded, the groups $\stabilizer S$ are separable~\cite[Thm~1.1]{PrzytyckiWise:cube_GraphManifold}, and double cosets $\stabilizer S\,\stabilizer S'$ are separable for intersecting surfaces $S,S'\in\widetilde{\mathcal S}^\mathrm{eff}$, by~\cite[Thm~1.2]{PrzytyckiWise:cube_GraphManifold}. Hence, by the above criterion, the complex $\pi_1M\backslash \mathcal{X}$ is virtually special.
\end{proof}

\section{Constructing flats from parallel families}\label{sec:flat_construct}
We now begin to establish statements needed in the proof of Theorem~\ref{thmi:cocompact_cubulation}. In this section we give methods of recognizing flats in a CAT(0) cube complex using parallel families, defined below.

\begin{defn}[Parallel family of lines]
Let $\mathcal S$ be a locally finite \emph{multiset} of geodesic lines in the Euclidean plane $\Euclidean^2$, i.e.\ a family of lines that are allowed to occur multiple times. We say that $\mathcal S$ is a \emph{parallel family} if the lines are pairwise parallel and for each $S\in \mathcal S$ each of the two components of $\Euclidean^2-S$ contains another line from $\mathcal S$. We say that $\mathcal S$ \emph{consists of $n$ parallel families} if $\mathcal S$ can be partitioned into $n$ parallel families that are maximal, or equivalently, pairwise transverse.
\end{defn}

\begin{defn}[Combinatorial flat]\label{defn:combinatorial_flat}
A \emph{combinatorial flat} $\mathbf E_n$ is a CAT(0) cube complex that is the standard tiling of $\Euclidean^n$ by unit $n$--cubes.
\end{defn}

\begin{rem}\label{rem:flat}
Let $\cuco$ be the CAT(0) cube complex from Construction~\ref{constr:Sageev}, dual to a collection $\mathcal{S}$ consisting of $n$ parallel families. If none of the lines in $\mathcal{S}$ coincide, then $\cuco$ is a combinatorial $n$--flat.
Otherwise $\cuco$ is the product of $n$ complexes, each of which is built from a sequence $(Q_k)_{k\in \integers}$ of cubes by gluing them along vertices so that the vertex $Q_{k-1}\cap Q_k$ is opposite to $Q_k\cap Q_{k+1}$ in $Q_k$. Each cube $Q_k$ corresponds to a collection of coinciding lines, and $\dimension Q_k$ is equal to the number of such lines.
\end{rem}

We would now like to recognize a combinatorial flat inside an ambient CAT(0) cube complex.

\begin{rem}\label{rem:plane_in_flat}
Any hyperplane $\mathfrak h$ in a CAT(0) cube complex $\cuco$ has a metric neighborhood of the form $\mathfrak h\times[-1,1]$. Hence for any CAT(0) geodesic $\varepsilon$ in $\cuco$ that is neither disjoint from nor contained in $\mathfrak h$, the intersection $\varepsilon\cap\mathfrak h$ is a point separating the intersections of $\varepsilon$ with the two open halfspaces of $\mathfrak h$. Consequently, for any isometrically embedded copy $\mathcal E\subset\cuco$ of $\Euclidean^2$ that is neither disjoint from nor contained in $\mathfrak h$, the intersection $\mathcal E\cap\mathfrak h$ is a line separating the convex intersections of $\mathcal E$ with these halfspaces.
\end{rem}

\begin{defn}[Parallel family of hyperplanes]\label{defn:parallel_family}
Let $\cuco$ be a \textnormal{CAT(0)} cube complex containing an isometrically embedded copy $\mathcal E$ of $\Euclidean^2$.
Let $\mathrm{Ess}(\mathcal E)$ denote the collection of hyperplanes that are neither disjoint from nor contain $\mathcal E$.
We say that $\mathcal E$ \emph{has $n$~parallel families} if the collection of lines $\mathfrak h\cap\mathcal E$, with $\mathfrak h\in\mathrm{Ess}(\mathcal E)$, consists of $n$ parallel families. The family of hyperplanes intersecting $\mathcal E$ along one of these families is called \emph{parallel} as well.
\end{defn}

Note that parallelism of $\mathfrak h\cap\mathcal E,\mathfrak h'\cap\mathcal E$ does not imply that the hyperplanes $\mathfrak h,\mathfrak h'$ are disjoint; these hyperplanes may intersect outside of $\mathcal E$.

Let $\cuco$ be finite-dimensional. If a group $J$ of automorphisms of $\cuco$ acts cocompactly on $\mathcal{E}$, then for each hyperplane $\mathfrak h\in \mathrm{Ess}(\mathcal E)$, the family of lines $\mathfrak h'\cap\mathcal E$, with $\mathfrak h'\in \mathrm{Ess}(\mathcal E)$, parallel to $\mathfrak h\cap\mathcal E$ is a parallel family. Consequently, there exists $n$ such that $\mathcal E$ has $n$ parallel families.

\begin{lem}\label{lem:straight}
Let $\cuco$ be a finite-dimensional \textnormal{CAT(0)} cube complex with a free action of a group $J=\integers^2$. Let $\mathcal E\subset\cuco$ be a $J$--cocompact isometrically embedded copy of $\Euclidean^2$. Then one of the following holds:
\begin{enumerate}
 \item For some $n\geq 3$, and some finite-index subgroup $J'\leq J$, the complex $\cuco$ contains a $J'$--invariant combinatorial $n$--flat $\mathcal Y$ such that $\mathcal Y^{(1)}\subset\cuco^{(1)}$ is isometrically embedded.
 \item The plane $\mathcal E$ has two parallel families.
\end{enumerate}
\end{lem}

\begin{proof}
Denote by $n$ the number of parallel families of $\mathcal E$. We first prove $n\geq 2$.
Each complementary component in $\mathcal E$ of the union of all the hyperplanes in $\mathrm{Ess}(\mathcal E)$ is contained in the cubical star of a vertex. Since each cube is bounded, each such component must be bounded. Hence there are hyperplanes $\mathfrak a,\mathfrak b\in \mathrm{Ess}(\mathcal E)$ with lines $\mathfrak a\cap\mathcal E$ and $\mathfrak b\cap\mathcal E$ intersecting transversely.

We shall now construct a combinatorial $n$--flat $\mathcal Y\subset\cuco$.
For a vertex $x\in\cuco$ and a hyperplane $\mathfrak h$, let $x(\mathfrak h)$ be the closed
halfspace of $\mathfrak h$ containing $x$.
If $\mathfrak h\notin\mathrm{Ess}(\mathcal E)$, so that $\mathcal E$ is contained in one or two of the closed halfspaces of $\mathfrak h$, we define $\mathcal E(\mathfrak h)$ to be one such fixed halfspace.
Consider the set $\mathcal Y^{(0)}$ of vertices $x\in \cuco$ satisfying the following conditions:
\begin{enumerate}
 \item [(a)] For $\mathfrak h\notin\mathrm{Ess}(\mathcal E)$, we have $x(\mathfrak h)=\mathcal E(\mathfrak h)$.
 \item [(b)] For $\mathfrak h,\mathfrak h'\in\mathrm{Ess}(\mathcal E)$, we have $x(\mathfrak h)\cap x(\mathfrak h')\cap\mathcal E\neq\emptyset$.
\end{enumerate}
Let $\mathcal Y$ be the subcomplex spanned by $\mathcal Y^{(0)}$, i.e.\ the union of the closed cubes all of whose vertices are contained in $\mathcal Y^{(0)}$. By condition (a), each vertex $x\in \mathcal Y^{(0)}$ is uniquely determined by the choice of halfspaces $x(\mathfrak h)$ with $\mathfrak h\in\mathrm{Ess}(\mathcal E)$ satisfying condition (b). By condition (b), for all $\mathfrak h$ from the same parallel family we have $x(\mathfrak h)\ni e$ for some $e\in \mathcal E$.
Thus $\mathcal Y$ is isomorphic to the CAT(0) cube complex dual to the intersection lines in $\mathcal E$ of $\mathrm{Ess}(\mathcal E)$. Hence $\mathcal Y$ is a combinatorial $n$--flat or possibly a complex described in Remark~\ref{rem:flat}.
A pair of points $e,e'\in \mathcal E$ is separated by $\mathfrak h\cap\mathcal E$ in $\mathcal E$ if and only if it is separated by $\mathfrak h$ in $\cuco$.
Hence $\mathcal Y^{(1)}\subset\cuco^{(1)}$ is isometrically embedded.

Let $J'$ be the finite-index subgroup of $J$ that acts trivially on the finite set of hyperplanes containing $\mathcal E$. The complex $\mathcal Y$ is $J'$--invariant since $J$ preserves parallel families and $J'$ preserves the collection of $\mathcal E(\mathfrak h)$. In the case where $\mathcal Y$ is not a combinatorial $n$--flat, we replace $J'$ with a further finite-index subgroup whose elements do not map a hyperplane in $\mathrm{Ess}(\mathcal E)$ to a distinct hyperplane with the same intersection line with $\mathcal E$. We can then replace $\mathcal Y$ with a $J'$--invariant subcomplex that is a combinatorial $n$--flat isometrically embedded on $1$--skeleton. If $n\geq 3$, case~(1) of the lemma is verified. Otherwise, we have $n=2$ and hence case~(2).
\end{proof}

The following complements Lemma~\ref{lem:straight}.

\begin{defn}[Combinatorial half-plane]\label{defn:comb_half_plane}
Let $\mathbf E_2$ be a combinatorial $2$--flat with a bi-infinite combinatorial geodesic $\gamma\subset\mathbf E_2^{(1)}$. The CAT(0) cube complex that is the closure of a component of $\mathbf E_2-\gamma$ is a \emph{combinatorial half-plane bounded by $\gamma$}.
\end{defn}

\begin{lem}\label{lem:non_cocompact_implies_E3}
Let $\cuco$ be a finite-dimensional \textnormal{CAT(0)} cube complex with a free action of a group $J=\integers^2$. Let $\mathcal E\subset\cuco$ be a $J$--cocompact isometrically embedded copy of $\Euclidean^2$ that has two parallel families. Then one of the following holds:
\begin{enumerate}
 \item For some finite-index subgroup $J'\leq J$, the complex $\cuco$ contains a $J'$--invariant subcomplex $\mathcal F\times\mathbf E_1$ where $\mathcal F$ is a combinatorial half-plane, such that $(\mathcal F\times\mathbf E_1)^{(1)}\subset\cuco^{(1)}$ is isometrically embedded.
 \item There exists a $J$--cocompact convex subcomplex $\widehat{\mathcal Y}\subset\cuco$.
\end{enumerate}
\end{lem}

We first reduce the proof of Lemma~\ref{lem:non_cocompact_implies_E3} to the following.

\begin{lem}\label{lem:far}
Let $\cuco$ be a finite-dimensional \textnormal{CAT(0)} cube complex and let $g\in\Aut(\cuco)$ act by a translation on a bi-infinite combinatorial geodesic $\gamma\subset \cuco^{(1)}$.  Then one of the following holds:
\begin{enumerate}

 \item There exists $m\geq 1$ and a $g^m$--invariant combinatorial half-plane $\mathcal F\subset \cuco$, bounded by $\gamma$ and contained in its cubical convex hull, such that $\mathcal F^{(1)}\subset\cuco^{(1)}$ is isometrically embedded.
 \item There exists $R$ such that if intersecting hyperplanes $\mathfrak h,\mathfrak h'$ intersect $\gamma$, then they intersect it in points at distance $\leq R.$
\end{enumerate}
\end{lem}

Here the \emph{cubical convex hull}  $\widehat{\mathcal Y}$ of a subcomplex $\mathcal Y$ of a CAT(0) cube complex $\cuco$ is the subcomplex of $\cuco$ spanned by the vertices that are not separated from $\mathcal Y^{(0)}$ by any hyperplane.

\begin{proof}[Proof of Lemma~\ref{lem:non_cocompact_implies_E3}]
Let $\mathcal Y$ be the combinatorial $2$--flat $\alpha \times\beta$ from Lemma~\ref{lem:straight}, where the factors come from the two parallel families.
We choose embeddings of the combinatorial geodesics $\alpha, \beta$ into $\mathcal Y$ with arbitrary fixed coordinates. As before, let $J'$ be a finite-index subgroup of $J$ preserving $\mathcal Y$. Consider the $J'$--invariant cubical convex hull $\widehat{\mathcal Y}$ of $\mathcal Y$, which is the product of the cubical convex hulls $\widehat\alpha$ and $\widehat\beta$ of $\alpha,\beta$. Let $\stabilizer_{J'} \alpha =\langle a \rangle, \ \stabilizer_{J'} \beta=\langle b \rangle$ and note that $\langle a,b \rangle \leq J'$ has finite index.

By Lemma~\ref{lem:far}, there are two possibilities. The first is that one of the subcomplexes $\widehat\alpha$ or $\widehat\beta$ contains a combinatorial half-plane $\mathcal F$, bounded by $\alpha$ or $\beta$, whose $1$--skeleton is isometrically embedded in that of $\widehat\alpha$ or $\widehat\beta$. Moreover, $\mathcal{F}$ is $a^m$-- or $b^m$--invariant for some $m>0$. Then after possibly replacing $J'$ with a further finite-index subgroup, we have that $\widehat{\mathcal Y}$ contains a $J'$--invariant subcomplex $\mathcal F\times\mathbf E_1$ isometrically embedded on $1$--skeleton. This verifies case~(1) of the lemma. The second possibility is that there exists $R$ such that any pair of intersecting hyperplanes from $\mathrm{Ess}(\mathcal E)$ intersects $\mathcal{Y}$ in lines at distance $\leq R$. Then $\widehat{\mathcal Y}$ contains finitely many $J'$--orbits of vertices, and hence $J'$ acts cocompactly on $\widehat{\mathcal Y}$.

In order to find a subcomplex that is $J$--invariant, we relax condition (a) from the proof of Lemma~\ref{lem:straight}, so that $x(\mathfrak h)$ can be arbitrary for $\mathfrak h$ containing $\mathcal E$. The cubical convex hull of such a set of vertices has the form $\widehat{\mathcal Y}\times I^k$, where $k$ is the number of such $\mathfrak h$, hence satisfies case~(2) of the lemma.
\end{proof}

The following completes the proof of Lemma~\ref{lem:non_cocompact_implies_E3}.

\begin{proof}[Proof of Lemma~\ref{lem:far}]
All the hyperplanes discussed in this proof are assumed to intersect $\gamma$.
Finite-dimensionality implies that there exists $m\geq 1$ such that any hyperplane $\mathfrak h$ is disjoint from and separates $g^{\pm m}\mathfrak h$. We order the hyperplanes so that $\mathfrak h<\mathfrak h'$ if the direction of the subpath of $\gamma$ from $\mathfrak h\cap \gamma$ to $\mathfrak h'\cap \gamma$ agrees with the direction of the translation $g$. Note that if $\mathfrak h<\mathfrak h''$ intersect, then each $\mathfrak h<\mathfrak h'<\mathfrak h''$ intersects $\mathfrak h$ or $\mathfrak h''$. In particular, if $\mathfrak h''=g^{mj}\mathfrak h'$, then $\mathfrak h'$ intersects $\mathfrak h$.

If case~(2) of the lemma does not hold, then there exist intersecting hyperplanes $\mathfrak h<\mathfrak h'$ with $\mathfrak h\cap \gamma,\ \mathfrak h'\cap \gamma$ at arbitrary distance. Since there are finitely many $\langle g^m\rangle$--orbits of hyperplanes, there exist hyperplanes $\mathfrak a<\mathfrak b$ such that $g^{-mi}\mathfrak a$ intersects $g^{mj}\mathfrak b$ for infinitely many $i,j\geq 0$. By the previous paragraph, $g^{-mi}\mathfrak a$ intersects $g^{mj}\mathfrak b$ for all $i,j\geq 0$.

We shall now partition the set of hyperplanes into two families $\mathfrak A\supset \{g^{mi}\mathfrak a\},\ \mathfrak B\supset\{g^{mj}\mathfrak b\}$, with $i,j\in \integers$, such that any $\mathfrak h\in\mathfrak A$ and $\mathfrak h'\in\mathfrak B$ satisfying $\mathfrak h<\mathfrak h'$ intersect.
For every hyperplane $\mathfrak h$, for sufficiently large $i,j\geq 0$ we have $g^{-mi}\mathfrak a<\mathfrak h<g^{mj}\mathfrak b$. Observe that if $\mathfrak h$ intersects $g^{mj}\mathfrak b$ and $g^{mj''}\mathfrak b$ for some $j<j''$, then $\mathfrak h$ intersects all $g^{mj'}\mathfrak b$ with $j< j'< j''$. Hence $\mathfrak h$ intersects all $g^{-mi}\mathfrak a$ for $i$ sufficiently large or all $g^{mj}\mathfrak b$ for $j$ sufficiently large. If the former property holds, or if both properties hold simultaneously, then let $\mathfrak h\in \mathfrak B$; otherwise let $\mathfrak h\in \mathfrak A$
and note that then $\mathfrak h$ is disjoint from all $g^{-mi}\mathfrak a$ for $i$ sufficiently large. Now for any $\mathfrak h\in\mathfrak A$ and $\mathfrak h'\in\mathfrak B$ satisfying $\mathfrak h<\mathfrak h'$, there is $i\geq 0$ with
$g^{-mi}\mathfrak a<\mathfrak h$. After possibly increasing $i$, the hyperplane $\mathfrak h'$ intersects $g^{-mi}\mathfrak a$, but $\mathfrak h$ does not intersect $g^{-mi}\mathfrak a$, whence $\mathfrak h$ intersects $\mathfrak h'$, as required.

Let $p=(x,x')$ be a pair of vertices of $\gamma$ such that either $x=x'$ or the direction of the subpath $xx'\subset \gamma$ agrees with the direction of the translation $g$. For every hyperplane $\mathfrak h$ we choose the halfspace $p(\mathfrak h)$ to be $x(\mathfrak h)$ if $\mathfrak h\in \mathfrak A$ or $x'(\mathfrak h)$ if $\mathfrak h\in \mathfrak B$. By the previous paragraph, any such halfspaces have non-empty intersection, and hence define a vertex in the cubical convex hull of $\gamma$. The set $\mathcal F^{(0)}$ of such vertices is $g^m$--invariant. Let $\mathcal F$ be the subcomplex spanned by $\mathcal F^{(0)}$. The $1$--skeleton of $\mathcal F$ is obtained by adding edges between vertices determined by $p$ that differ by moving $x$ through a hyperplane in $\mathfrak A$ or $x'$ through a hyperplane in $\mathfrak B$. The squares of $\mathcal F$ are obtained by performing two such moves simultaneously. It is easy to see that $\mathcal F$ is a combinatorial half-plane bounded by $\gamma$ and that $\mathcal F^{(1)}\subset \cuco^{(1)}$ is isometrically embedded.
\end{proof}

\section{Cocompactly cubulated graph manifolds}\label{sec:cocompact_implies_flip}
In this section, we prove Theorem~\ref{thmi:cocompact_cubulation}. As in Section~\ref{sec:flip_implies_cocompact}, we denote by $M$ a graph manifold with universal cover $\widetilde M$ and Basse--Serre tree $\Gamma$. We pull back a fixed Riemannian metric on $M$ to $\widetilde{M}$.  Finally, we assume that $G=\pi_1 M$ acts freely and cocompactly on a CAT(0) cube complex $\cuco$.

\subsection{Straight subgroups}

\begin{lem}\label{lem:no_E3}
The inclusion $\widetilde T_e\subset\widetilde M$ does not extend to a $G_e'$--equivariant quasi-isometric embedding $\phi\colon\widetilde T_e\times[0,\infty)\rightarrow\widetilde M$ for any finite-index subgroup $G_e'\leq G_e$.
\end{lem}

\newcommand{\cone}{\mathbf{Cone}}

\begin{proof}
Let $e=(v,v')$. If such $\phi$ exists, then let $L$ be its additive constant, so that the $L$--neighborhood of the image under $\phi$ of any path-connected set is path-connected. We first show that for any such $\phi$ and for any $s\in [0,\infty)$ there is $t\in \widetilde T_e$ with $\phi(t,s)\in \neb_L(\widetilde B_v\cup \widetilde B_{v'})$. Otherwise, fix any $t\in \widetilde T_e$. Without loss of generality $\phi(t,s)$ lies in a component of $\widetilde M-\neb_L(\widetilde B_v)$ that does not contain $\widetilde B_{v'}$. Let $g\in G_e'$ be non-central in $G_v$. Then $g$ does not stabilize any of the edges of $\Gamma$ incident to $v$ except for $e$. Thus $\phi(t,s)$ and $\phi(gt,s)$ lie in different components of $\widetilde M-\neb_L(\widetilde B_v)$. Consider any path $\tau$ in $\widetilde T_e$ joining $t$ with $gt$. Then $\neb_L(\phi(\tau,s))$ joins two different components of $\widetilde M-\widetilde B_v$ but is disjoint from $\widetilde B_v$. Thus $\neb_L(\phi(\tau,s))$ is disconnected, which is a contradiction. Since $G'_e$ acts cocompactly on $\widetilde T_e$, it follows that the entire image of $\phi$ is contained in a uniform neighborhood of $\widetilde B_v\cup \widetilde B_{v'}$.

The quasi-isometric embedding $\phi$ induces a bi-Lipschitz embedding of asymptotic cones $\cone(\widetilde T_e\times[0,\infty))\rightarrow\cone(\widetilde B_v\cup \widetilde B_{v'})$. Here we choose an arbitrary ultrafilter, scaling constants, and observation points in $\widetilde T_e\times\{0\}$. Then $\cone(\widetilde T_e\times[0,\infty))=\Euclidean^2\times[0,\infty)$ which has topological dimension $3$. On the other hand $\cone(\widetilde B_v\cup \widetilde B_{v'})$ is obtained by gluing two copies of the product of an $\reals$--tree with $\Euclidean^1$ along $\Euclidean^2$; the result has dimension $2$, a contradiction.
\end{proof}

\begin{defn}[Essential hyperplane, essential core]\label{defn:essential_hyperplane_core}
Let $Y$ be a subspace of a metric space $X$. We say that $Z\subset X$ is \emph{$Y$--essential} if $Z$ is not contained in any $k$--neighborhood $\neb_k(Y)$ of $Y$.

Let $\mathcal V$ be a subspace of a CAT(0) cube complex $\cuco$. Let $\mathrm{Ess}_\cuco(\mathcal V)$ be the set of hyperplanes $\mathfrak h$ of $\cuco$ such that the intersections of $\mathcal V$ with both halfspaces of $\mathfrak h$ are $\mathfrak h$--essential. We usually omit the subscript $\cuco$, and this agrees with the notation $\mathrm{Ess}(\mathcal E)$, which we have used so far.

Similarly, if $H$ acts on $\cuco$, a hyperplane $\mathfrak h$ is called \emph{$H$--essential} if the intersections of an (hence any) $H$--orbit with both halfspaces of $\mathfrak h$ are $\mathfrak h$--essential. If $\mathcal V\subset \cuco$ is $H$--cocompact, then $\mathrm{Ess}(\mathcal V)$ is the set of $H$--essential hyperplanes.

If $\mathcal V$ is a convex subcomplex, then each $\mathfrak h\in\mathrm{Ess}_\cuco(\mathcal V)$ intersects $\mathcal V$, these intersections form $\mathrm{Ess}_{\mathcal V}(\mathcal V)$, and thus we can identify $\mathrm{Ess}_\cuco(\mathcal V)$ with $\mathrm{Ess}_{\mathcal V}(\mathcal V)$ and omit the subscript.
The \emph{essential core} $\mathcal{V}^{ess}$ is the CAT(0) cube complex dual to $\mathrm{Ess}(\mathcal V)$, which is a quotient of $\mathcal V$. If a group $H$ acts on $\mathcal V$ freely and cocompactly, then it acts on $\mathcal {V}^{ess}$ freely and cocompactly as well. If $J\leq H$ and a hyperplane $\mathfrak h$ of $\cuco$ is $J$--essential, then we have $\mathfrak h\in\mathrm{Ess}(\mathcal V)$ and the corresponding hyperplane in $\mathcal{V}^{ess}$ is $J$--essential as well.
%(see \cite{CapraceSageev:rank_rigidity}).
\end{defn}

\begin{defn}[Straight subgroup]\label{defn:characteristic_element}
Let $e$ be an edge of $\Gamma$. A maximal cyclic subgroup of $G_e$ is \emph{straight} if it contains a nontrivial element stabilizing a $G_e$--essential hyperplane of $\cuco$.
\end{defn}

\begin{prop}\label{prop:straight_edge}
For each edge $e$ of $\Gamma$, the group $G_e$ has exactly two straight subgroups.
\end{prop}

\begin{proof}
By~\cite[Thm~1.4]{Haglund:cubes_semisimple}, $G$ acts semisimply on $\cuco$. By the Flat Torus Theorem~\cite[Thm~II.7.1]{BridsonHaefliger:book}, $\cuco$ contains an isometrically embedded copy $\mathcal E$ of $\Euclidean^2$ on which $G_e$ acts cocompactly. We apply Lemma~\ref{lem:straight} with $J=G_e$. Note that the set of $G_e$--essential hyperplanes coincides with the union of parallel families $\mathrm {Ess}(\mathcal E)$. In case (1) of Lemma~\ref{lem:straight}, a $G$--equivariant quasi-isometry $\cuco\rightarrow \widetilde M$ maps $\mathcal Y$ to $\widetilde{T}_e\times (-\infty,\infty)^{n-2}$ contradicting Lemma~\ref{lem:no_E3}.
Hence we have case (2), which says that $\mathcal{E}$ has two parallel families. Consider the two maximal cyclic subgroups $\integers_a,\integers_b< G_e$ stabilizing lines in these families. Then $\integers_a,\integers_b$ are straight, since they have finite-index subgroups stabilizing all hyperplanes in one of the corresponding families of $\mathrm {Ess}(\mathcal E)$.
On the other hand, each $G_e$--essential hyperplane intersects $\mathcal E$ along one of these lines, so these are the only straight subgroups.
\end{proof}

\subsection{Cocompact cores}

\begin{lem}\label{lem:cocompact_cores_for_JSJ}
For each edge $e$ of $\Gamma$, there exists a convex $G_e$--cocompact subcomplex $\widehat{\mathcal Y}_e\subset\cuco$.
\end{lem}

\begin{proof}
As in the proof of Proposition~\ref{prop:straight_edge}, the Flat Torus Theorem yields a $G_e$--cocompact copy $\mathcal E$ of $\Euclidean^2$ isometrically embedded in $\cuco$. We apply Lemma~\ref{lem:straight} as before and conclude that we have case (2), hence we can apply Lemma~\ref{lem:non_cocompact_implies_E3}. Case (1) of that lemma leads to a contradiction with Lemma~\ref{lem:no_E3}. Hence we have case (2) saying that there exists a $G_e$--cocompact convex subcomplex $\widehat{\mathcal Y}_e\subset\cuco$.
\end{proof}

\begin{prop}\label{prop:core_for_block}
For each vertex $v$ of $\Gamma$, there exists a convex $G_v$--cocompact subcomplex $\mathcal V_v\subset\cuco$.
\end{prop}

\newcommand{\leftx}{\overleftarrow{\cuco}}
\newcommand{\rightx}{\overrightarrow{\cuco}}

Given a subcomplex ${\mathcal Y}\subset\cuco$, let ${\mathcal Y}^{+0}=\mathcal Y$ and let the \emph{cubical $k$--neighborhood} ${\mathcal Y}^{+k}$ be the union of all closed cubes intersecting ${\mathcal Y}^{+(k-1)}$. If $\mathcal Y\subset \cuco$ is convex, then ${\mathcal Y}^{+k}$ is convex as well.

\begin{proof}
Fix a base vertex $x\in \cuco$. By Lemma~\ref{lem:cocompact_cores_for_JSJ}, there is a $G$--equivariant
family of $G_e$--cocompact convex subcomplexes
$\widehat{\mathcal Y}_e\subset \cuco$, where $e$ varies over the set of edges incident to $v$.
Let $\phi\colon \widetilde{M}\rightarrow \cuco$ be a $G$--equivariant
quasi-isometry. Each $\widetilde{M}-\neb_k(\widetilde{T}_e)$ has two
$\widetilde{T}_e$--essential components. Let $\overleftarrow M_e^k$ be the $\widetilde{T}_e$--essential component disjoint from $\widetilde{B}_v$ and let $\overrightarrow M_e^k=\widetilde M-\overleftarrow M_e^k$. In other words, the space $\widetilde{M}$ has two \emph{poles} with
respect to each $\widetilde{T}_e$ (see \cite[Appendix~A]{CapracePrzytycki:bipolar}). Poles are quasi-isometry invariants \cite[Lem~A.2]{CapracePrzytycki:bipolar}, so by
\cite[Lem~A.4]{CapracePrzytycki:bipolar}, there is $k$ such that $\cuco-\widehat{\mathcal Y}^{+k}_e$ has two
$\widehat{\mathcal Y}_e$--essential components, one of which, denoted by $\leftx^k_e$, is
disjoint from the orbit $G_vx\subset \cuco$. Moreover, there exists $K$ such
that $\phi (\overleftarrow M_e^K)\subset \leftx^k_e$; see~\cite[Sublem~A.3]{CapracePrzytycki:bipolar}. Let $\rightx^k_e=\cuco-\leftx^k_e$.
Let $\mathcal V_v\subset \cuco$ be the convex $G_v$--invariant subcomplex $\bigcap_e
\rightx^k_e$, which contains $x$. Let $K'\geq K$ be such that $\neb_{K'}(\widetilde
B_v)$ contains all complementary components of $\neb_{K}(\widetilde T_e)$ that are not $\widetilde T_e$--essential.
Then we have $\phi^{-1}(\mathcal V_v)\subset \bigcap_e \overrightarrow
M^K_e\subset \neb_{K'}(\widetilde B_v)$. Since $\neb_{K'}(\widetilde
B_v)$ is $G_v$--cocompact, the complex $\mathcal V_v$ is $G_v$--cocompact as
well.
\end{proof}

The final preparatory lemma requires using two results from~\cite{CapraceSageev:rank_rigidity}.

\begin{lem}\label{lem:essential_core} Consider the product of the free cyclic and a finitely generated free non-abelian group $H=\integers\times\mathbb{F}$. Suppose that $H$ acts freely and cocompactly on a \textnormal{CAT(0)} cube complex $\mathcal V$. Then the following hold:
\begin{enumerate}
 \item The essential core $\mathcal V^{ess}$ of $\mathcal V$ is a product $\mathcal V_a\times\mathcal V_b$, where $\mathcal V_a,\mathcal V_b$ are unbounded.
 \item The group $H$ has a finite-index subgroup $H'=H_a\times H_b$ that preserves the above decomposition, where $H_a$ acts trivially on $\mathcal V_b$ and $H_b$ acts trivially on $\mathcal V_a$.
 \item We have $H_a=\integers\cap H'$ and the group $H_b$ embeds as a finite-index subgroup of the free group $H/\integers\cong\mathbb{F}$ under the natural quotient.
\end{enumerate}
\end{lem}

\begin{proof}
Since $H$ is a direct product with infinite factors, none of its elements is rank 1 in the action on $\mathcal V^{ess}$.
Hence \cite[Cor~6.4(iii)]{CapraceSageev:rank_rigidity} implies assertion~(1). By \cite[Prop~2.6]{CapraceSageev:rank_rigidity} the group $\Aut(\mathcal V^{ess})$ contains $\Aut(\mathcal V_a)\times\Aut(\mathcal V_b)$ as a finite-index subgroup, from which assertion~(2) follows. Since each of the subcomplexes $\mathcal V_a,\mathcal V_b$ is unbounded, $H_a,H_b$ are non-trivial.

Any finite-index subgroup $H'\leq H$ has the form $\integers'\times\mathbb F'$, where $\integers'=\integers\cap H'$. We will identify one of $H_a,H_b$ as $\integers'$.
First note that either $H_a$ or $H_b$ has trivial center. Otherwise we would have $\integers^2$ in the center of $\integers'\times\mathbb F'$. Secondly, consider $z\in \integers'=\integers\cap H'$. Then $z\in H_a$ or $z\in H_b$, since otherwise the projections of $z$ to $H_a,H_b$ would both be central. Without loss of generality we assume $\integers'\leq H_a$. Conversely, an element $h\in H_a$ commutes with both $\integers'$ and $H_b$. Since the only elements with non-cyclic centralizer in $\integers'\times\mathbb F'$ are in $\integers'$, we have $h\in \integers'$, establishing $H_a=\integers'$. Hence $H_b$ embeds in the quotient $H/\integers$. Finally, $H_b$ has finite index in $H/\integers$ since $H'$ has finite index in $H$. This establishes assertion~(3).
\end{proof}

\subsection{Vanishing of charges}\label{subsec:charges_vanish}

\begin{proof}[Proof of Theorem~\ref{thmi:cocompact_cubulation}]
Suppose $G=\pi_1M$ acts freely and cocompactly on a CAT(0) cube complex $\cuco$. By~\cite[Prop~4.4]{LW}, after passing to a finite cover we can assume that the blocks $B_{\bar{v}}$ of $M$ have no singular fibers. In other words their base orbifolds $F_{\bar{v}}$ are surfaces and consequently their fundamental groups are free. Hence $G_v=\integers_v\times \mathbb F_v$, where $\integers_v$ is the cyclic fiber group and $\mathbb F_v$ is a free group.

By Proposition~\ref{prop:core_for_block}, there is a convex $G_v$--cocompact subcomplex $\mathcal V_v\subset \cuco$. We consider
the essential core $\mathcal V^{ess}_v$ of $\mathcal V_v$, which decomposes as $\mathcal V_a\times\mathcal V_b$ by Lemma~\ref{lem:essential_core}(1) applied to $H=G_v$. Let $H'=H_a\times H_b$ be the subgroups provided by Lemma~\ref{lem:essential_core}(2).

Let $\bar v$ be the image of $v$ in $\overline \Gamma$ and let $F'_{\bar v}\rightarrow F_{\bar v}$ be the finite cover of the base surface $F_{\bar v}$ corresponding to the embedding $H_b\leq G_v/ \integers_v\cong\pi_1 F_{\bar v}$ coming from Lemma~\ref{lem:essential_core}.(3). Let $\mathcal C$ be the collection of boundary curves of $F'_{\bar v}$. For each $c\in\mathcal C$, choose an edge $e$ of $\Gamma$ incident to $v$ so that $G_e\cap H_b< H_b\cong \pi_1F'_{\bar v}$ is a conjugate of $\pi_1 c$.
 The inclusion $G_{e}\cap H_b< G_{e}$ can be considered as inclusion on homology
$\homology_1(G_{e}\cap H_b)< \homology_1(G_{e})=\homology_1(T_{\bar e})$. Choose an embedded non-vertical curve $\breve c\subset T_{\bar e}$ and an integer $n_{c}$ such that $n_{c}[\breve c]$ is the image of $[c]\in\homology_1(G_{e}\cap H_b)$ under this inclusion. Let $\phi \colon \homology_1(H_b)\rightarrow\homology_1(G_v)=\homology_1(B_{\bar v})$ be induced by $H_b< G_v$.
Hence $\phi \big(\sum_{c\in\mathcal C}[c]\big)=\sum_{c\in\mathcal C}n_{c}[\breve c]$. Since $\mathcal C$ is the collection of boundary curves of a surface, we have $\sum_{c}[c]=0$ in $\homology_1(H_b)$, whence $\sum_{c}n_{c}[\breve c]=0$ in $\homology_1(B_{\bar v})$.
To show that $M$ is chargeless, that is to verify the analogous formula from Definition~\ref{def:chargeless}, it remains to show that $\breve c$ is homotopic to $Z^{\bar e}_{\bar v'}$.

\begin{claim}
Let $\integers_a,\integers_b< G_e$ be the straight subgroups coming from Proposition~\ref{prop:straight_edge}. Then the groups $H_a, \ G_e\cap H_b$ are finite-index subgroups of $\integers_a,\integers_b$ or vice-versa.
\end{claim}

Before we justify the claim, we explain how to use it to complete the proof. The claim yields $H_a\leq \integers_a$ and since $H_a\leq\integers_v$ by Lemma~\ref{lem:essential_core}(3), we obtain $\integers_a=\integers_v$. Moreover, the claim gives that the image of $[c]\in\homology_1(G_{e}\cap H_b)$ in $\homology_1(G_{e})=G_e$ lies in $\integers_b$, whence $\langle[\breve c]\rangle=\integers_b$. If $e=(v,v')$ and we apply the claim to $e$ and $v'$, we similarly obtain that either $\integers_a$ or $\integers_b$ coincides with $\integers_{v'}$. Since $\integers_{v'}\cap\integers_v=\{1\}$, we have $\integers_b=\integers_{v'}$. For Definition~\ref{def:chargeless} the circle $Z^{\bar e}_{\bar v'}$ was defined to satisfy $\langle[Z^{\bar e}_{\bar v'}]\rangle=\integers_{v'}$ in $\homology_1 (T_{\bar e})=G_e$. Thus we have $\langle[\breve c]\rangle=\integers_b=\integers_{v'}=\langle[Z^{\bar e}_{\bar v'}]\rangle$, which ends the proof of the theorem.

It remains to justify the claim. By Lemma~\ref{lem:essential_core}(3) we have $H_a=\integers_v\cap H'$ and hence $G_e\cap H'=H_a\times (G_e\cap H_b)$, which is a product of cyclic groups.
Thus it suffices to find nontrivial elements in $\integers_a\cap H_a,\ \integers_b\cap (G_e\cap H_b)$.
Since $\integers_a$ is straight, it contains a nontrivial element $h$ stabilizing a $G_e$--essential hyperplane of $\cuco$.
Since $H'\leq H$ has finite index, after passing to a power we can assume $h\in H_a\times (G_e\cap H_b)$. Since $G_e< G_v$, the corresponding hyperplane in the $G_v$--essential core $\mathcal V^{ess}_v$ of $\mathcal V_v$ is $G_e$--essential as well. Without loss of generality assume that $h$ stabilizes a hyperplane $\mathcal V_a\times \mathfrak h$ where $\mathfrak h$ is a hyperplane in $\mathcal V_b$. Then $h=a\times b$, where $b$ stabilizes $\mathfrak h$. Since $1\times b$ and $H_a\times 1$ stabilize the $G_e$--essential hyperplane $\mathcal V_a\times \mathfrak h$, they cannot generate a finite-index subgroup of $G_e$, whence $b=1$. Thus $h\in H_a$, as desired, and using the same argument we find a nontrivial element in $\integers_b\cap (G_e\cap H_b)$.
\end{proof}

\bibliographystyle{alpha}
\bibliography{graph_man_cocompact_refs}

\end{document}